\documentclass[10pt,reqno]{amsart}

\usepackage{amsmath,amssymb,amsfonts,amsthm}
\usepackage{enumerate}
\usepackage{url}
\usepackage[breaklinks=true,colorlinks=true,linkcolor=blue,citecolor=blue!60!green,filecolor=blue,urlcolor=blue]{hyperref}
\hypersetup{linktocpage}
\usepackage{bm}
\usepackage[hiresbb]{graphicx,xcolor}
\usepackage{tcolorbox}
\usepackage{booktabs,tabularx}
\usepackage{mathtools}
\mathtoolsset{showonlyrefs=true}
\usepackage{mathrsfs}
\usepackage{float}
\usepackage{subcaption}
\usepackage{yfonts}
\usepackage{accents}
\usepackage{here}
\usepackage{hhline}
\allowdisplaybreaks[3]
\usepackage[top=30truemm,bottom=30truemm,left=30truemm,right=30truemm]{geometry}

\makeatletter
\newcommand{\tpmod}[1]{{\@displayfalse\pmod{#1}}}
\makeatother
\def\rnum#1{\expandafter{\romannumeral #1}} 
\def\Rnum#1{\uppercase\expandafter{\romannumeral #1}}
\makeatletter
\let\amsmath@bigm\bigm
\renewcommand{\bigm}[1]{%
\ifcsname fenced@\string#1\endcsname
\expandafter\@firstoftwo
\else
\expandafter\@secondoftwo
\fi
{\expandafter\amsmath@bigm\csname fenced@\string#1\endcsname}%
{\amsmath@bigm#1}%
}
\newcommand{\DeclareFence}[2]{\@namedef{fenced@\string#1}{#2}}
\makeatother
\DeclareFence{\mid}{|}

\DeclareMathOperator{\SL}{SL}

\DeclareMathOperator{\PSL}{PSL}

\newcommand{\N}{\mathbb{N}}

\newcommand{\R}{\mathbb{R}}

\newcommand{\Z}{\mathbb{Z}}

\newcommand{\bsm}{\left(\begin{smallmatrix}}
\newcommand{\esm}{\end{smallmatrix} \right)}
\newcommand{\bpm}{\begin{pmatrix}}
\newcommand{\epm}{\end{pmatrix}}

\newtheorem{theorem}{Theorem}[section]

\newtheorem{conjecture}[theorem]{Conjecture}

\newtheorem{observation}{Observation}

\theoremstyle{definition}

\newtheorem{remark}[theorem]{Remark}

\numberwithin{equation}{section}

\newcommand{\Tr}{\mathrm{Tr}}
\newcommand{\li}{\mathrm{li}}

\newcommand{\vol}{\mathrm{vol}}

\usepackage{newtxtext,newtxmath}
\usepackage{eucal}

\begin{document}

\title{Spectral Exponential Sums on Hyperbolic Surfaces}

\author{Ikuya Kaneko}
\address{Department of Mathematics, California Institute of Technology, 1200 E California Blvd, Pasadena, CA 91125, USA}
\email{ikuyak@icloud.com}
\urladdr{\href{https://sites.google.com/view/ikuyakaneko/}{https://sites.google.com/view/ikuyakaneko/}}
\thanks{The author is supported in part by the Masason Foundation and the Spirit of Ramanujan STEM Talent Initiative.}



\date{\today}

\dedicatory{}

\begin{abstract}%%%%%%%%%%%%%%%%%%%%%%%%%%%%%%%%%%%%%%%%%%%%%%%%%%%%%%%
We study an exponential sum over Laplacian eigenvalues $\lambda_{j} = 1/4+t_{j}^{2}$ with $t_{j} \leqslant T$ for Maa{\ss} cusp~forms on $\Gamma \backslash \mathbb{H}$, where $\Gamma$ is a cofinite Fuchsian group acting on the upper half-plane $\mathbb{H}$. The~aim~is~to establish an asymptotic formula which expresses spectral exponential sums in terms of an oscillatory component, von Mangoldt-like functions and Selberg zeta functions. The behaviour is determined by whether $\Gamma$ is essentially cuspidal or not.
\end{abstract}

\maketitle

\section{Introduction}\label{introduction}%%%%%%%%%%%%%%%%%%%%%%%%%%%%%%%%%%%%%%%%%%%

\subsection{Motivation}\label{motivation}%%%%%%%%%%%%%%%%%%%%%%%%%%%%%%%%%%%%%%%%%%%
This article introduces a new idea to prove an asymptotic law for the spectral exponential sum
\begin{equation}\label{spectral-exponential-sum}
\mathcal{S}(T, X) = \sum_{t_{j} \leqslant T} X^{it_{j}}
\end{equation}
in the spectral aspect. Here $\lambda_{j} = s_{j}(1-s_{j}) = 1/4+t_{j}^{2}$ are non-exceptional eigenvalues of the hyperbolic~Laplacian acting on $\mathcal{L}^{2}(\Gamma \backslash \mathbb{H})$ for a congruence subgroup $\Gamma \subseteq \SL_{2}(\Z)$ and the upper half-plane $\mathbb{H}$. We henceforth~exploit~the sign convention $t_{j} > 0$. There are various applications of spectral exponential sums, but we elaborate on the~Prime Geodesic Theorem for ease of exposition. This concerns~the asymptotic behaviour of the counting function
\begin{equation*}
\pi_{\Gamma}(X) = \#\{\{P \}: N(P) \leqslant X \},
\end{equation*}
where $\{P \}$ signifies a primitive hyperbolic conjugacy class in $\Gamma$ and $N(P)$ stands for its norm. We call~$\{P \}$~primitive if a hyperbolic element $P \in \Gamma$ cannot be expressed as $Q^{j}$ with $j \geqslant 2$ for some $Q \in \Gamma$; hence every hyperbolic conjugacy class is a power of some primitive class. By partial summation, one passes between asymptotic results for $\pi_{\Gamma}(X)$ and the allied counting function
\begin{equation*}
\Psi_{\Gamma}(X) = \sum_{N(P) \leqslant X} \Lambda_{\Gamma}(P),
\end{equation*}
where the sum runs over all hyperbolic classes and $\Lambda(\cdot)$ is the analogue of the von Mangoldt function defined by $\Lambda_{\Gamma}(P) = \log N(P_{0})$ if $\{P \}$ is a power of a primitive hyperbolic class $\{P_{0} \}$.

In his seminal work, Iwaniec~\cite{Iwaniec1984} obtained the explicit formula for $\Psi_{\Gamma}(X)$ as a corollary of the Selberg trace formula and an inequality \`{a} la Brun--Titchmarsh; see~\cite{Bykovskii1994,Iwaniec1984,KanekoKoyama2018}. For any congruence~subgroup $\Gamma$, we have that
\begin{equation}\label{explicit-formula}
\Psi_{\Gamma}(X) = X+\sum_{1/2 < s_{j} < 1} \frac{X^{s_{j}}}{s_{j}}
 + \sum_{|t_{j}| \leqslant T} \frac{X^{s_{j}}}{s_{j}}+O \left(\frac{X}{T}(\log X)^{2} \right)
\end{equation}
for an auxiliary parameter $1 \leqslant T \leqslant X^{1/2}(\log X)^{-2}$. The first sum on the right-hand side arises from the~exceptional eigenvalues of the Laplacian and the sum $\sum_{|t_{j}| \leqslant T}$ denotes that it is symmetrised by including both $t_{j}$ and $-t_{j}$. The explicit formula~\eqref{explicit-formula} implies that even when $T$ is suitably controlled, one cannot reach an error term smaller than $O(X^{3/4+\epsilon})$ without considering any cancellation in the spectral exponential sum. It behoves us to mention that the corresponding barrier for $\mathcal{S}(T, X)$ is $O(T^{2})$, which is called the trivial bound.

We recall that the Selberg zeta function is built out of prime geodesics as follows:
\begin{equation}\label{Selberg-zeta-function}
Z_{\Gamma}(s) = \prod_{\{P_{0} \}} \prod_{k = 0}^{\infty} (1-N(P_{0})^{-s-k}),
\end{equation}
where the outer product runs over all primitive hyperbolic classes. Given the analogue of the Riemann~hypothesis for the Selberg zeta functions of congruence subgroups (apart from a finite number of exceptional zeroes), we should conjecture that $\Psi_{\Gamma}(X) = X+O(X^{1/2+\epsilon})$. This remains an impenetrable open problem. For convenience, we define $\eta$ to be such that the following formula holds:
\begin{equation}\label{eta}
\pi_{\Gamma}(X) = \li(X)+O_{\epsilon}(X^{\eta+\epsilon}), \qquad \li(X) = \int_{2}^{X} \frac{dt}{\log t}.
\end{equation}
In broad strokes, the optimal value of $\eta$ so far established marks the level of current technology.

If $\Gamma$ is arithmetic, an improvement over the 3/4-barrier can be deduced by appealing to the Kuznetsov~formula. A general rule of thumb is that using the Selberg trace formula leads to cleaner formul{\ae} than using the Kuznetsov formula --- for estimations the latter is known to be more beneficial nonetheless. In retrospect, the first landmark triumph to go beyond the 3/4-barrier was due to Iwaniec~\cite{Iwaniec1984}, who showed for $\Gamma = \PSL_{2}(\Z)$ that
\begin{equation*}
\mathcal{S}(T, X) \ll_{\epsilon} TX^{11/48+\epsilon}
\end{equation*}
with $T, \, X \gg 1$. This bound was improved in the celebrated work of Luo--Sarnak~\cite{LuoSarnak1995}:
\begin{equation*}
\mathcal{S}(T, X) \ll X^{1/8} T^{5/4}(\log T)^{2}.
\end{equation*}
This is also achievable for any arithmetic group $\Gamma \subset \PSL_{2}(\R)$ such as a congruence subgroup and a subgroup arising from a quaternion devision algebra. They showed $\eta = 7/10$ according to the optimal choice $T = X^{3/10}$. The crucial step in all of these works was to showcase a nontrivial bound on $\mathcal{S}(T, X)$. Soundararajan--Young~\cite{SoundararajanYoung2013} succeeded in showing the best known bound for $\PSL_{2}(\Z)$, namely $\eta = 2/3+\theta/6$ with~$\theta$ a subconvex exponent for quadratic Dirichlet $L$-functions. The Weyl-strength exponent $\theta = 1/6$ of Conrey--Iwaniec~\cite{ConreyIwaniec2000} yields~$\eta = 25/36$. The conjectural exponent $\eta = 2/3$ follows from the Lindel\"{o}f hypothesis for Dirichlet $L$-functions. The author~\cite{Kaneko2019} recently established an exact connection between pointwise and second moment bounds for the error term in~\eqref{eta}.

It is important to speculate what the correct order of $\mathcal{S}(T, X)$ should be. Petridis--Risager~\cite[Conjecture 2.2]{PetridisRisager2017} have conjectured square root cancellation, namely
\begin{equation}\label{Petridis-Risager}
\mathcal{S}(T, X) \ll_{\epsilon} T(TX)^{\epsilon}.
\end{equation}
Moreover, the bound~\eqref{Petridis-Risager} yields not only the best possible error term $O(X^{1/2+\epsilon})$ in the Prime Geodesic~Theorem, but also the best error term on average for the hyperbolic lattice point counting. In the appendix of~\cite{PetridisRisager2017}, Laaksonen has proven via the Selberg trace formula that the conjecture~\eqref{Petridis-Risager} is true for a fixed $X > 1$ in the spectral aspect.

\subsection{Results}\label{results}%%%%%%%%%%%%%%%%%%%%%%%%%%%%%%%%%%%%%%%%%%%%%%
This work considers the congruence subgroups $\Gamma_{0}(q), \, \Gamma_{1}(q)$ and $\Gamma(q)$ defined in Section~\ref{scattering-determinants}. Just before stating our results, we introduce the following standard notation. Let $\Lambda(X)$ be the von Mangoldt function extended to $\R$ by defining it to be 0 when $X$ is not a prime power. Define an analogous function $\Lambda_{\Gamma}(X)$ for norms of hyperbolic conjugacy classes in $\Gamma$ as
\begin{equation*}
\Lambda_{\Gamma}(X) \coloneqq
	\begin{cases}
	\log N(P_{0}) & \text{if $X = N(P)^{j}, \, j \geqslant 1$},\\
	0 & \text{otherwise}.\\
	\end{cases}
\end{equation*}
One of our aims is to establish that the spectral exponential sum $\mathcal{S}(T, X)$ associated to a congruence subgroup $\Gamma$ obeys a conjectural bound in the spectral aspect to which we have already alluded.
\begin{theorem}\label{main}
For $\Gamma = \Gamma_{0}(q), \, \Gamma_{1}(q)$, $\Gamma(q)$, we define
\begin{equation*}
S(T) = \frac{1}{\pi} \arg Z_{\Gamma} \left(\frac{1}{2}+iT \right) \quad \text{and} \quad
\mathcal{G}(T) = \int_{1/2}^{2} \log|Z_{\Gamma}(\sigma+iT)| d\sigma.
\end{equation*}
For a fixed $X > 1$, we then have
\begin{equation}\label{Fujii}
\mathcal{S}(T, X) = \frac{\vol(\Gamma \backslash \mathbb{H})}{2\pi i \log X} X^{iT} T
 + \frac{T}{2\pi}(X^{1/2}-X^{-1/2})^{-1} \Lambda_{\Gamma}(X)
 + \frac{T}{\pi} X^{-1/2} \Lambda(X^{1/2}) \sideset{}{^{\star}} \sum_{\psi} \Re(\psi(X^{1/2}))
 + X^{iT} S(T)+O(\mathcal{G}(T)),
\end{equation}
where the sum runs over Dirichlet characters $\psi$ to some modulus dividing $q$ and the meaning of $\star$ on the sum is described in Lemma~\ref{Hejhal-Huxley}.
\end{theorem}
The formula~\eqref{Fujii} for $\Gamma = \PSL_{2}(\Z)$ was announced without a proof by Fujii in 1984 in his report~\cite{Fujii1984}.~Note~that his theorem recovers Laaksonen's result. In Fujii's work, the additional term $X^{iT} S(T)$ was replaced~by~$O(T/\log T)$, which follows from the trivial bounds $S(T) \ll T/\log T$ and $\mathcal{G}(T) \ll T/\log T$. We may revisit the issue~of~bounding $\mathcal{G}(T)$ in a cleverer manner, elsewhere. Theorem~\ref{main} generalises his result to the congruence subgroups~$\Gamma_{0}(q)$, $\Gamma_{1}(q)$ and $\Gamma(q)$. It makes us believe that the extremely strong bound $\Psi_{\Gamma}(X)-X \ll X^{1/2+\epsilon}$ may hold.
\begin{remark}\label{remark}
For a fixed $X > 1$, the spectral exponential sum $\mathcal{S}(T, X)$ has a peak of order $T$ whenever $X$ is equal to a power of a norm of a primitive hyperbolic class in $\Gamma$, or to an even power of a prime number which comes from the determinant of the scattering matrix associated with $\Gamma$. At the peak point $X$ corresponding to the peak of $\mathcal{S}(T, X)$, the second or the third term in the asymptotic formula~\eqref{Fujii} does not vanish. Indeed, $\mathcal{S}(T, X)$ does~not always exhibit a clear peak structure; see~\cite[Appendix]{PetridisRisager2017}.
\end{remark}
Remark~\ref{remark} for $\Gamma = \PSL_{2}(\Z)$ is in accordance with the theorems by Chazarain~\cite{Chazarain1974}, where for the wave~kernel the singularities occur at the lengths of closed geodesics. Via Theorem~\ref{main}, one can utilise $\mathcal{S}(T, X)$ as a~roundabout way of detecting $N(P)$. There must be a connection underneath the surface between spectral parameters~$t_{j}$~and~the length spectrum in $\mathbb{H}$. The spectral exponential sum $\mathcal{S}(T, X)$ differs from the classical one counted with Riemann zeroes where the peaks are known to be at all prime powers. By Theorem~\ref{main}, we conjecture the following~square root cancellation for $\mathcal{S}(T, X)$, which generalises the conjecture of Petridis--Risager for $\Gamma = \PSL_{2}(\Z)$.
\begin{conjecture}\label{conjecture}
Let $X > 2$. For any arithmetic hyperbolic surface, the spectral exponential sum $\mathcal{S}(T, X)$ exhibits square root cancellation in $T$ with uniform dependence on $X$ up to a factor of $X^{\epsilon}$, namely
\begin{equation*}
\mathcal{S}(T, X) \ll T^{1+\epsilon} X^{\epsilon}.
\end{equation*}
\end{conjecture}

\section*{Acknowledgements}%%%%%%%%%%%%%%%%%%%%%%%%%%%%%%%%%%%%%%%%%%%%%%%%
Thanks are owed to Shin-ya Koyama for pointing out errors in an earlier version of this article.

\section{Sketch of requisites}\label{sketch-of-requisites}%%%%%%%%%%%%%%%%%%%%%%%%%%%%%%%%%%%
We provide background materials which we shall need later to establish Theorem~\ref{main}.

\subsection{Weyl's law}\label{Weyls-law}%%%%%%%%%%%%%%%%%%%%%%%%%%%%%%%%%%%%%%%%%%%
The spectrum of the Laplacian on hyperbolic surfaces has a connection with certain objects~in number theory such as $L$-functions and exponential sums. There are a number of conjectures centered~around~the structure of the discrete spectrum. Apart from $\lambda_{0} = 0$, we cannot separately count the discrete and the~continuous spectra. For a general cofinite group, it remains an open problem which spectrum is larger in order of~magnitude; see Section~\ref{counterexamples}. Weyl's law renders asymptotic behaviour of both discrete and continuous spectrum in~an~expanding window. For any Fuchsian group $\Gamma$ of the first kind, it asserts that
\begin{equation}\label{Weyl-law}
N_{\Gamma}(T)+M_{\Gamma}(T) \sim \frac{\vol(\Gamma \backslash \mathbb{H})}{4\pi} T^{2}
\end{equation}
as $T \to \infty$, where $N_{\Gamma}(T) \coloneqq \# \{j: t_{j} \leqslant T \}$ and $M_{\Gamma}(T)$ is the winding number which accounts for the contribution of the continuous spectrum:
\begin{equation*}
M_{\Gamma}(T) \coloneqq \frac{1}{4\pi} \int_{-T}^{T} -\frac{\varphi^{\prime}}{\varphi} \left(\frac{1}{2}+it \right) dt.
\end{equation*}
Here $\varphi$ signifies the determinant of the scattering matrix $\Phi$ of $\Gamma$. In order to establish the asymptotic formula~\eqref{Weyl-law}, the complete spectral decomposition of an automorphic kernel is needed. The Selberg trace formula is a beneficial tool for deducing~\eqref{Weyl-law}. If $\Gamma \backslash \mathbb{H}$ is compact, $M_{\Gamma}(T)$ vanishes and the identity~\eqref{Weyl-law} can be reduced to the asymptotic behaviour of $N_{\Gamma}(T)$. We note that $M_{\Gamma}(T)$ is real and approximately equal to the number of scattering poles~on~the left of the critical line $\Re(s) = 1/2$ with heights at most $T$.

As we shall see later, the formula~\eqref{Weyl-law} is not satisfactory for our purpose. We recall a refined version of~Weyl's law (cf.~\cite[Theorem 5.2.1]{Venkov1982},~\cite[Theorem 7.2]{Venkov1990},~\cite[(11.3)]{Iwaniec2002}), which is available for the general cofinite scenario:
\begin{theorem}\label{Weyl}
Let $\Gamma$ be a cofinite subgroup of $\PSL_{2}(\R)$. We then have that
\begin{equation}\label{Weyl-law-2}
N_{\Gamma}(T)+M_{\Gamma}(T) = \frac{\vol(\Gamma \backslash \mathbb{H})}{4\pi} T^{2}
 - \frac{h}{\pi} T \log T+c_{\Gamma} T+S(T)+\mathrm{const.}+w(T)
\end{equation}
as $T \to \infty$, where $h$ is the number of inequivalent cusps and $c_{\Gamma}$ is a certain constant dependent only on $\Gamma$, $S(T)$~is the same as in the introduction, and $w(T)$ can be chosen such that $w^{\prime}(T) \ll T^{-2}$.
\end{theorem}
Theorem~\ref{Weyl} can be deduced almost verbatim the argument used to establish~\eqref{Weyl-law}.

\subsection{Scattering determinants}\label{scattering-determinants}%%%%%%%%%%%%%%%%%%%%%%%%%%%%%
In this subsection, we contemplate the congruence subgroups
\begin{align*}
\Gamma_{0}(q) &\coloneqq \bigg\{\begin{pmatrix} a & b \\ c & d \end{pmatrix} \in \PSL_{2}(\Z): c \equiv 0 \tpmod q \bigg\},\\
\Gamma_{1}(q) &\coloneqq \bigg\{\begin{pmatrix} a & b \\ c & d \end{pmatrix} \in \PSL_{2}(\Z): 
a, d \equiv 1 \tpmod q, \ c \equiv 0 \tpmod q \bigg\},\\
\Gamma(q) &\coloneqq \bigg\{\begin{pmatrix} a & b \\ c & d \end{pmatrix} \in \PSL_{2}(\Z): 
a, d \equiv 1 \tpmod q, \ b, c \equiv 0 \tpmod q \bigg\}.
\end{align*}
Recall that
\begin{equation*}
\vol(\Gamma \backslash \mathbb{H}) = \frac{\pi}{3} [\PSL_{2}(\Z): \Gamma] = 
	\begin{cases}
	\dfrac{\pi}{3} q \displaystyle{\prod_{p \mid q}} \left(1+\dfrac{1}{p} \right) & \text{for $\Gamma = \Gamma_{0}(q)$},\\
	\dfrac{\pi}{3} q^{2} \displaystyle{\prod_{p \mid q}} \left(1-\dfrac{1}{p^{2}} \right) & \text{for $\Gamma = \Gamma_{1}(q)$},\\
	\dfrac{\pi}{3} q^{3} \displaystyle{\prod_{p \mid q}} \left(1-\dfrac{1}{p^{2}} \right) & \text{for $\Gamma = \Gamma(q)$}.
	\end{cases}
\end{equation*}
For $\Gamma = \Gamma_{0}(q), \, \Gamma_{1}(q)$, or $\Gamma(q)$, an accurate calculation of $\varphi$ was executed by Hejhal and Huxley for squarefree $q$ and for every $q$, respectively. For notational convenience, we only record the result of Huxley.
\begin{theorem}[Huxley~\cite{Huxley1984}]\label{Hejhal-Huxley}
Let $\Gamma = \Gamma_{0}(q), \, \Gamma_{1}(q)$, or $\Gamma(q)$. Then the scattering determinant of $\Gamma$ is given by
\begin{equation}\label{scattering-determinant}
\varphi(s) = (-1)^{(h-h_{0})/2} \left(\frac{\Gamma(1-s)}{\Gamma(s)} \right)^{h} \left(\frac{A}{\pi^{h}} \right)^{1-2s}
\sideset{}{^{\star}} \prod_{\psi} \frac{L(2-2s, \overline{\psi})}{L(2s, \psi)}.
\end{equation}
Notation is as follows: $h$ equals the number of inequivalent cusps, and $h_{0}$ is an integer for which
\begin{equation*}
\Tr(\Phi(s)) \to -h_{0} \quad \text{as} \quad s \to \frac{1}{2}.
\end{equation*}
The Dirichlet characters $\psi$ appearing in the product can be expressed as
\begin{equation*}
\psi(n) = \psi_{1}(n) \psi_{2}(n) \omega_{m_{1} m_{2}}(n) \quad \text{for} \quad n \in \N,
\end{equation*}
where $\psi_{\ell} \ (\ell = 1, 2)$ is a primitive Dirichlet character modulo $q_{\ell}$ and $\omega_{m_{1} m_{2}}$ is the trivial character modulo $m_{1} m_{2}$. As for the product over $\psi$, the variables $q_{1}$, $q_{2}$, $m_{1}$, and $m_{2}$ are over all positive integers satisfying the conditions below, and $\psi_{\ell} \ (\ell = 1, 2)$ runs over all possible Dirichlet characters.
\begin{enumerate}
\item[(a)] $(m_{1}, m_{2}) = 1, \, m_{1} q_{1} \mid q, \, m_{2} q_{2} \mid q$ for $\Gamma = \Gamma(q)$,
\item[(b)] (a) and $m_{1} = 1, \, q_{1} \mid m_{2}$ for $\Gamma = \Gamma_{1}(q)$,
\item[(c)] (b) and $q_{1} = q_{2}, \, \psi_{1} = \psi_{2}$ for $\Gamma = \Gamma_{0}(q)$.
\end{enumerate}
The product in~\eqref{scattering-determinant} has $h$ terms and $A$ is a positive integer composed of primes dividing $q$:
\begin{equation*}
A = 
	\begin{cases}
	\displaystyle{\prod_{(a)} m_{1} m_{2} q_{1} q} & \text{for $\Gamma = \Gamma(q)$},\\
	\displaystyle{\prod_{(b)} q_{1} q} & \text{for $\Gamma = \Gamma_{1}(q)$},\\
	\displaystyle{\prod_{(c)} \frac{q_{1} q}{(m_{2}, q/m_{2})}} & \text{for $\Gamma = \Gamma_{0}(q)$}.
	\end{cases}
\end{equation*}
\end{theorem}

\section{Proof of Theorem~\ref{main}}\label{proof-of-theorem-1.1}%%%%%%%%%%%%%%%%%%%%%%%%%%%%%%
This section is aimed at proving Theorem~\ref{main} with the tools in \S\ref{sketch-of-requisites} in mind.

\begin{proof}[Proof of Theorem~\ref{main}]
Using Theorem~\ref{Weyl}, one obtains
\begin{align}\label{spectral-exponential-sum}
\begin{split}
\mathcal{S}(T, X) &= \int_{1}^{T} X^{it} dN_{\Gamma}(t)\\
& = \frac{\vol(\Gamma \backslash \mathbb{H})}{2\pi} 
\int_{1}^{T} X^{it} t dt-\int_{1}^{T} X^{it} dM_{\Gamma}(t)+\int_{1}^{T} X^{it} dS(T)+O(\log T)\\
& = \mathscr{L}^{1}+\mathscr{L}^{2}+\mathscr{L}^{3}+O(\log T).
\end{split}
\end{align}
We first handle $\mathscr{L}^{1}$~for~which we use integration by parts to deduce
\begin{equation*}
\mathscr{L}^{1} = \frac{\vol(\Gamma \backslash \mathbb{H})}{2\pi i} \frac{X^{iT}}{\log X} T+O(1).
\end{equation*}
As for the second term $\mathscr{L}^{2}$, we calculate the scattering determinant via Theorem~\ref{Hejhal-Huxley}:
\begin{equation*}
-\frac{\varphi^{\prime}}{\varphi} \left(\frac{1}{2}+it \right)
 = 2\log \frac{A}{\pi^{h}}+h\frac{\Gamma^{\prime}}{\Gamma} \left(\frac{1}{2} \pm it \right)
 + 2\sideset{}{^{\star}} \sum_{\psi} \left(\frac{L^{\prime}}{L}(1-2it, \overline{\psi})+\frac{L^{\prime}}{L}(1+2it, \psi) \right).
\end{equation*}
Hence we exploit the Stirling asymptotics, deriving
\begin{equation}\label{M}
M_{\Gamma}(t) = \frac{1}{4\pi i} \log \sideset{}{^{\star}} \prod_{\psi} 
\left(\frac{L(1+2it, \psi) L(1+2it, \overline{\psi})}{L(1-2it, \psi) L(1-2it, \overline{\psi})} \right)+O(t \log t).
\end{equation}
We use integration by parts again and then bounding trivially the Dirichlet $L$-functions in~\eqref{M} yields
\begin{align*}
\mathscr{L}^{2} &= -\frac{1}{4\pi i} \sideset{}{^{\star}} \sum_{\psi} \int_{1}^{T} X^{it} 
d\log \left(\frac{L(1+2it, \psi) L(1+2it, \overline{\psi})}{L(1-2it, \psi) L(1-2it, \overline{\psi})} \right)+O(\log T)\\
& = \frac{\log X}{4\pi} \sideset{}{^{\star}} \sum_{\psi} \int_{1}^{T} 
X^{it}(\log(L(1+2it, \psi) L(1+2it, \overline{\psi}))-\log(L(1-2it, \psi) L(1-2it, \overline{\psi}))) dt+O(\log T).
\end{align*}
Letting $\delta = (\log X)^{-1}$, the first integration involving $\log(L(1+2it, \psi) L(1+2it, \overline{\psi}))$ equals
\begin{align*}
\frac{1}{2i} \int_{1+2i}^{1+2iT} X^{(t-1)/2} \log(L(t, \psi) L(t, \overline{\psi})) dt\\
&\hspace{-5cm} = \frac{1}{2i} \left(\int_{1+\delta+2i}^{1+\delta+2iT}-\int_{1+2iT}^{1+\delta+2iT}
 + \int_{1+2i}^{1+\delta+2i} \right) X^{(t-1)/2} \log(L(t, \psi) L(t, \overline{\psi})) dt\\
&\hspace{-5cm} = \frac{1}{2i} \int_{1+\delta+2i}^{1+\delta+2iT} X^{(t-1)/2} \log(L(t, \psi) L(t, \overline{\psi})) dt+O(\log T)\\
&\hspace{-5cm} = X^{\delta/2} \sum_{n = 2}^{\infty} \frac{(\psi(n)+\overline{\psi}(n)) \Lambda(n)}{n^{1+\delta} \log n} 
\int_{1}^{T} \exp(it(\log X-2\log n)) dt+O(\log T)\\
&\hspace{-5cm} = \frac{2T}{\log X} X^{-1/2}(\psi(X^{1/2})+\overline{\psi}(X^{1/2})) \Lambda(X^{1/2})+O(\log T),
\end{align*}
where $\Lambda(X)$ is the classical von Mangoldt function. The second integral involving $\log(L(1-2it, \psi) L(1-2it, \overline{\psi}))$ is obviously bounded, since the corresponding integrand becomes $\exp(it(\log X+2\log n))$.

It remains to deal with the third term $\mathscr{L}^{3}$. By partial integration, we have
\begin{equation*}
\mathscr{L}^{3} = X^{iT} S(T)-i\log X \int_{1}^{T} X^{it} S(T) dt+O(1)
 = X^{iT} S(T)+\mathscr{L}^{4} \log X-i\mathscr{L}^{5} \log X+O(1),
\end{equation*}
where
\begin{equation*}
\mathscr{L}^{4} = \int_{1}^{T} \sin(t \log X) S(T) dt \quad \text{and} \quad \mathscr{L}^{5} = \int_{1}^{T} \cos(t \log X) S(T) dt.
\end{equation*}
For the Selberg zeta function in~\eqref{Selberg-zeta-function}, we define
\begin{equation*}
F(z) = \log Z_{\Gamma}(z) \sin \left(\left(\frac{1}{2}-z \right)i \log X \right),
\end{equation*}
where we choose the principal value of the logarithm and the branch of $\log Z_{\Gamma}(z)$ is taken such that $\log Z_{\Gamma}(z)$~is~real for every $z > 1$. Mimicking the treatment of $\mathscr{L}^{2}$, we consider the rectangle with vertices $1+\delta+i, \, 1+\delta+iT, \, 1/2+iT$ and $1/2+i$ with $\delta = (\log X)^{-1}$. Hence, it follows that
\begin{equation*}
\mathscr{L}^{4} %= \frac{1}{\pi} \int_{1}^{T} \sin(t \log X) \Im(Z_{\Gamma}(1+it)) dt 
 = \Im \left(\frac{1}{\pi i} \int_{1/2+i}^{1/2+iT} F(z) dz \right)\\
 = \Im \left(\frac{1}{\pi i} 
\left(\int_{1+\delta+i}^{1+\delta+iT}-\int_{1/2+iT}^{1+\delta+iT}+\int_{1/2+i}^{1+\delta+i} \right) F(z) dz \right).
\end{equation*}
The third integral is bounded. From the definition~\eqref{Selberg-zeta-function}, we compute the first integral as
\begin{align*}
&\frac{1}{\pi} \int_{1}^{T} \log Z_{\Gamma}(1+\delta+it) \sin \left(\left(t-\left(\frac{1}{2}+\delta \right)i \right) \log X \right) dt\\
& = -\frac{1}{2\pi i} \sum_{\{P \}} \frac{\Lambda_{\Gamma}(P) (1-N(P)^{-1})^{-1}}{N(P)^{1+\delta} \log N(P)} 
\int_{1}^{T} \exp \left(i \left(t-\left(\frac{1}{2}+\delta \right)i \right) \log X-it \log N(P) \right) dt\\
&\hspace{0.4cm} + \frac{1}{2\pi i} \sum_{\{P \}} \frac{\Lambda_{\Gamma}(P) (1-N(P)^{-1})^{-1}}{N(P)^{1+\delta} \log N(P)} 
\int_{1}^{T} \exp \left(i \left(t-\left(\frac{1}{2}+\delta \right)i \right) \log X-it \log N(P) \right) dt\\
& = -\frac{X^{1/2+\delta}}{2\pi i} \sum_{\{P \}} 
\frac{\Lambda_{\Gamma}(P) (1-N(P)^{-1})^{-1}}{N(P)^{1+\delta} \log N(P)} \int_{1}^{T} \exp(it(\log X-\log N(P))) dt+O(1)\\
& = -\frac{\Lambda_{\Gamma}(X)}{2\pi iX^{1/2} \log X} (X^{1/2}-X^{-1/2})^{-1} T+O(1),
\end{align*}
whence we obtain
\begin{equation*}
\Im \left(\frac{1}{\pi i} \int_{1+\delta+i}^{1+\delta+iT} F(z) dz \right)
 = \frac{T}{2\pi \log X} X^{-1/2}(X^{1/2}-X^{-1/2})^{-1} \Lambda_{\Gamma}(X)+O(1).
\end{equation*}
Finally, the second integral can be bounded as
\begin{equation*}
-\frac{1}{\pi i} \int_{1/2+iT}^{1+\delta+iT} F(z) dz
 = \frac{1}{\pi} \int_{1/2}^{1+\delta} \log Z_{\Gamma}(\sigma+iT) \sinh((\sigma-1+iT) \log X) d\sigma
 \ll \mathcal{G}(T),
\end{equation*}
where $\mathcal{G}(T)$ is the same as in the introduction. The integral $\mathscr{L}^{5}$ is estimated in a similar manner, namely~$\mathscr{L}^{5} \ll \mathcal{G}(T)$. Collecting those estimates concludes the proof~of~Theorem~\ref{main}.
\end{proof}

\section{Counterexamples}\label{counterexamples}%%%%%%%%%%%%%%%%%%%%%%%%%%%%%%%%%%%%%

\subsection{Phillips--Sarnak theory}\label{Phillips-Sarnak-theory}%%%%%%%%%%%%%%%%%%%%%%%%%%%%%%%
We are interested in the validity of the asymptotic law $N_{\Gamma}(T) \sim \vol(\Gamma \backslash \mathbb{H}) T^{2}/4\pi$. In a major breakthrough, Phillips--Sarnak~\cite{PhillipsSarnak1985-2} innovated a new methodology~of the real analytic deformation of discrete groups in $\PSL_{2}(\R)$. They examined the behaviour of a Maa{\ss} cusp~form for $\Gamma_{0}(p)$ under quasi-conformal deformations $\Gamma_{\tau}$ with $0 \leqslant \tau \leqslant 1$. The folklore theory of Phillips--Sarnak manifests that Maa{\ss} cusp forms~are~rare whose existence should be restricted to certain arithmetic groups. The work of Luo~\cite{Luo2001} yields that Weyl's law
\begin{equation*}
N_{\Gamma_{\tau}}(T) \sim \frac{\vol(\Gamma \backslash \mathbb{H})}{4\pi} T^{2}
\end{equation*}
cannot hold for generic $\Gamma_{\tau}$ if one assumes that the eigenvalue multiplicities for $\Gamma_{0}(p) \backslash \mathbb{H}$ are bounded.

\begin{observation}\label{essentially-cuspidal}
Non-arithmetic groups should not be essentially cuspidal and there should only be a finite~number of Maa{\ss} cusp forms in such a case. If the latter assertion is true, the bound $\mathcal{S}(T, X) \ll 1$ holds.
\end{observation}

We say that $\Gamma$ is \textit{essentially cuspidal} if $N_{\Gamma}(T)$ dominates the behaviour of $N_{\Gamma}(T)+M_{\Gamma}(T)$, namely if
\begin{equation*}
N_{\Gamma}(T) \sim \frac{\vol(\Gamma \backslash \mathbb{H})}{4\pi} T^{2}.
\end{equation*}
Here we consider the eigenvalue multiplicity problem for arithmetic groups with an emphasis on the case~of~$\Gamma_{0}(q)$. The case of $q = 1$ was first considered by Cartier~\cite{Cartier1971} and consequently he conjectured that the cuspidal spectrum is simple based on limited numerical data. Nowadays, intensive numerical computations of Then~\cite{Then2005} are~available, determining the first 53000 eigenvalues for $\PSL_{2}(\Z) \backslash \mathbb{H}$.

For $q \geqslant 2$, the situation is entirely different, since the cuspidal spectrum for $\Gamma_{0}(q) \backslash \mathbb{H}$ is not necessarily simple, which follows from the existence of newforms and oldforms having the same Laplace eigenvalue.~Nevertheless, one can ask whether the new part of the cuspidal spectrum consisting of eigenvalues for $\Gamma_{0}(q) \backslash \mathbb{H}$ associated~to newforms is simple. This conjecture should be true for squarefree $q$, which is manifested in the work of~Bolte--Johansson~\cite{BolteJohansson1999,BolteJohansson1999-2} and Str\"{o}mbergsson~\cite{Strombergsson2001}. We also refer the reader to the recent work of Humphries~\cite{Humphries2019}.

Since the estimate $S(T) \ll T/\log T$ holds for all cofinite surfaces, our proof of Theorem~\ref{main} works for more extensive context, although the integral entailing the winding number remains. One then derives from~\eqref{spectral-exponential-sum} that
\begin{theorem}\label{explicit-formula-for-the-spectral-exponential-sum}
Let $\Gamma$ be a cofinite subgroup of $\PSL_{2}(\R)$. For a fixed $X > 1$, we have as $T \to \infty$ that
\begin{equation}\label{generic}
\mathcal{S}(T, X)  = \frac{\vol(\Gamma \backslash \mathbb{H})}{2\pi i} \frac{X^{iT}}{\log X} T
 + \frac{T}{2\pi}(X^{1/2}-X^{-1/2})^{-1} \Lambda_{\Gamma}(X)-\int_{1}^{T} X^{it} dM_{\Gamma}(t)+O \left(\frac{T}{\log T} \right).
\end{equation}
\end{theorem}

If $\Gamma$ is essentially cuspidal, the spectral exponential sum $\mathcal{S}(T, X)$ asymptotically equals the first term in~\eqref{generic}.~As an important instance, if $\Gamma \backslash \mathbb{H}$ is compact, then the aforementioned formula can be reduced to
\begin{equation*}
\mathcal{S}(T, X) = \frac{\vol(\Gamma \backslash \mathbb{H})}{2\pi i} \frac{X^{iT}}{\log X} T
 + \frac{T}{2\pi}(X^{1/2}-X^{-1/2})^{-1} \Lambda_{\Gamma}(X)+O \left(\frac{T}{\log T} \right).
\end{equation*}

%\bibliography{zeta}
%\bibliographystyle{abbrv}

\end{document}